\documentclass{article}
\usepackage[utf8]{inputenc} 
\usepackage[T1]{fontenc}    
\usepackage{hyperref}       
\usepackage{url}            
\usepackage{booktabs}       
\usepackage{amsfonts}       
\usepackage{nicefrac}       
\usepackage{microtype}      
\usepackage{graphicx}
\graphicspath{{figures/}}
\usepackage{ifthen}
\usepackage{tikz}
\usepackage{eso-pic}
\usepackage{color}
\usepackage{float}
\usepackage{multicol}
\usepackage{setspace}
\usepackage{mathptmx}
\usepackage[theorems]{tcolorbox}
\usepackage{mdframed}
\usepackage{amsmath,amsthm,amssymb,mathtools,flexisym,bbm}
\usepackage{xcolor,paralist,hyperref,fancyhdr,etoolbox}
\usepackage{mathrsfs}
\usepackage{pgfplots}
\usepackage{xcolor}
\usepackage[shortlabels]{enumitem}
\usepackage{algorithm}
\usepackage[noend]{algpseudocode}
\usepackage{cite}
\usepackage{multirow}
\usepackage[font=scriptsize]{caption}
\usepackage{subcaption}
\usepackage{comment}
\usepackage{xspace}
\usepackage{environ}
\usepackage{longtable}
\usepackage{indentfirst,csquotes}
\usepackage{array}

\usepackage[margin=1in]{geometry}

\newtheorem{theorem}{Theorem}[section]

\newtheorem{example}[theorem]{Example}

\newtheorem{proposition}[theorem]{Proposition}

\newtheorem{remark}[theorem]{Remark}

\newcolumntype{H}{>{\setbox0=\hbox\bgroup}c<{\egroup}@{}}
\hypersetup{hypertexnames=false,colorlinks=true,allcolors=black}

\newcommand{\R}{\mathbb{R}}

\newcommand{\E}{\mathbb{E}}

\DeclareMathOperator{\pes}{pes}

\newcommand{\st}{\text{s.t.}}

\allowdisplaybreaks            
\AtBeginDocument{%
  \setlength{\abovedisplayskip}{4pt}
  \setlength{\belowdisplayskip}{4pt}
  \setlength{\abovedisplayshortskip}{4pt}
  \setlength{\belowdisplayshortskip}{4pt}
}

\tikzset{every picture/.style={line width=0.75pt}} 
\title{Dynamic Interval Scheduling with Random Start and End Times
} 
\author{Rui Gong\thanks{
Corresponding author (rgong44@gatech.edu)\\
\indent~~ Georgia Institute of Technology, Atlanta, GA 30332, USA\\
\indent~~ Partially funded by the U.S.\ Office of Naval Research, N00014-23-1-2631, and the U.S.\ Air Force Office of Scientific Research, FA9550-25-1-0341.} \and 
Alejandro Toriello}

\begin{document}
\maketitle

\begin{abstract}
\noindent
We study sequential interval scheduling with random task start and end times. Task weights and the discrete distributions of task intervals are known, but each task's realized interval is revealed only upon commitment;
this also eliminates tasks that conflict with the committed task.
The objective is to maximize the expected weight of a conflict-free schedule.
We propose two models that differ in how conflicts are enforced, develop relaxations and bounds for each, and present a computational study.
\end{abstract}

\section{Introduction}
The \emph{interval scheduling problem} is classical in operations research, industrial engineering and computer science.
It considers a set of tasks $N:=[n] = \{1, \dotsc, n\}$ and consecutive time slots $M:=[m]$.
Each task is represented by an interval $[s_i,e_i]$ describing the time slots it occupies if scheduled.
The \emph{maximum interval scheduling problem} seeks a set of pairwise nonoverlapping intervals (a schedule) of maximum total weight.
When all tasks have the same weight, the greedy rule that repeatedly selects the earliest-ending task and discards conflicts is optimal;
the weighted version is also solvable in $O(n\log n)$ time \cite{Kleinberg+Tardos:06a}.

An \emph{interval graph} is an undirected graph representing intervals on a line, where each vertex corresponds to an interval and an edge connects two vertices if their intervals overlap.
A feasible schedule corresponds exactly to an independent set of the associated interval graph.
Since interval graphs are perfect, the maximum interval scheduling problem can be represented by the clique primal-dual LP pair of the interval graph $G=(N,E)$,
\begin{multicols}{2}
    \noindent
    \begin{align}
\label{LP-P}
\tag{LP-P}
\begin{split}
    \max_{x\geq 0}\quad &w^\top x\\
    \st\quad &\sum_{i\in C}x_i\leq 1, \forall C\in\mathcal{C}(G)
\end{split}
\end{align}
\columnbreak
\begin{align}
\label{LP-D}
\tag{LP-D}
\begin{split}
    \min_{\mu\geq0}\quad &\sum_{C\in\mathcal{C}(G)}\mu_C\\
    \text{s.t.}\quad &\sum_{C\ni i}\mu_C\geq w_i,~\forall i\in N,
\end{split}
\end{align}
\end{multicols}
\noindent
where $\mathcal{C}(G)$ is the set of all maximal cliques of $G$.
\eqref{LP-P} formulates the maximum independent set problem and \eqref{LP-D} formulates the corresponding minimum clique cover problem.
For arbitrary perfect graphs, \eqref{LP-P} and \eqref{LP-D} are tight formulations, {although a perfect graph may have exponentially many maximal cliques.}
In contrast, in interval graphs the number of maximal cliques is at most $\min\{m, n\}$.
Thus the deterministic maximum interval scheduling problem can be solved efficiently either directly or via these LP relaxations.
Our goal is to extend this setting to incorporate uncertainty and sequential decisions while retaining tractable relaxations and algorithms.

Several interval scheduling models under uncertainty have been studied.
A classical one is \emph{online interval scheduling}~\cite{Lipton1994OnlineIS}, where intervals of fixed length are presented to the scheduler in the order of their start times and must be accepted or rejected irrevocably.
The objective is to maximize total weight while maintaining feasibility.
Recent work includes online interval scheduling with predictions of upcoming intervals~\cite{boyar2025onlineintervalschedulingpredictions} and models that allow overlaps but maximize a satisfaction objective~\cite{KOBAYASHI2020673}.
Other variants include interval scheduling with random delays~\cite{BRANDA2024106542} and settings where tasks may depart unexpectedly and service times are stochastic~\cite{xu2024stochasticschedulingabandonmentsgreedy}.

We study the setting where the set of tasks and their weights are known in advance but task start and end times are random.
The scheduler commits to tasks sequentially, learns realized intervals as tasks are committed, and seeks to maximize the expected total weight.

\section{Dynamic Scheduling with Random Start and End Times}
\label{section:DSRSE}

We again let $M=[m]$ be the set of slots, $N=[n]$ the set of tasks, and $w\in\R_{+}^{N}$ the task weights.
The task set is partitioned into $N=N_{\text{int}}\mathbin{\cup}N_{\text{pt}}$, where the former set represents interval tasks, and the latter denotes the special case of singleton tasks. For an interval task $i\in N_{\mathrm{int}}$, the start and end times $\widetilde s_i\sim D_i^s$ and $\widetilde e_i\sim D_i^e$ are independent, with supports $S_i,E_i$ satisfying $\max S_i\leq\min E_i$, and $I_i=[\widetilde s_i,\widetilde e_i]$. A singleton task $i\in N_{\mathrm{pt}}$ has one random slot $\widetilde s_i\sim D_i^s$ with support $S_i$, and occupies $I_i=\{\widetilde s_i\}$; it has no separate end-time distribution. Thus, the task occupies exactly one slot, although $S_i$ may contain several slots. All random inputs are independent across tasks. We let $Q_i$ be the induced distribution of $I_i$ and write $D=(Q_i)_{i\in N}$. 

Initially, all tasks are available but their realized intervals are unknown to the scheduler. When the scheduler accepts an available task $i$, this decision is irrevocable; the scheduler earns $w_i$ and observes $I_i$.
Every other available task $j$ with $I_j\cap I_i\neq\emptyset$ is deleted. A remaining task's interval is not directly revealed, but remaining certifies that the interval avoids every previously accepted task; therefore, the task's interval distribution is updated to condition on this information. 

{A solution is a non-anticipatory policy: at each decision stage, the policy chooses an available task using only previously revealed information, including previously scheduled tasks, revealed intervals, and the identities of surviving and deleted tasks. Because any conflicting task is eliminated when a task is scheduled, the tasks scheduled by any feasible policy are pairwise non-overlapping.
The objective is to maximize the expected total weight of scheduled tasks}. We call this problem \emph{dynamic scheduling with random start and end times (DSRSE)}.

One motivating application is the scheduling of jobs from different agents at a shared computing or testing facility. An agent knows the exact time interval required for its task, but the scheduler may initially only know the task interval's distribution based on previous usage data. An agent may not wish to reveal exact times beforehand if they are part of a confidential operating plan. A security or confidentiality protocol can therefore require a committed acceptance before those times are disclosed; for example, confidentiality concerns in shared-resource scheduling are studied in \cite{SinghOKeefe2016}. Under the protocol considered here, after each acceptance the remaining agents or an authorized checking service provide exact compatibility feedback, and an incompatible request must withdraw.  
An agent can thus know its own time interval while withholding this information from the scheduler, even if disclosure might improve its acceptance probability. However, compatibility feedback can itself reveal information, and thus our model does not make a formal privacy guarantee.

{For a formal dynamic programming formulation, let $U\subseteq N$ be the available tasks and let $O\subseteq M$ be the union of the slots that previously accepted intervals occupy. Conditioned on the observed history, the intervals of the available tasks remain independent, with conditional distributions
}\[
{Q_i^O(J):=\Pr(I_i=J\mid I_i\cap O=\emptyset)
=\frac{Q_i(J)\mathbf1\{J\cap O=\emptyset\}}{\sum_{K:K\cap O=\emptyset}Q_i(K)},\qquad i\in U;
}\]
the denominator is positive in a reachable state.
If task $i$ is accepted and $I_i=J$ is revealed, each other task $j\in U\setminus\{i\}$ survives independently with probability
\[
\rho_j(O,J)=\sum_{K:K\cap J=\emptyset}Q_j^O(K).
\]
{Let $R\subseteq U\setminus\{i\}$ be the resulting survivor set. The next state is $(R,O\cup J)$, with survivor distributions $Q_j^{O\cup J}$. The Bellman recursion is
}\[
{V(U,O)=\max_{i\in U}\left\{w_i+
\sum_J Q_i^O(J)\,\E\bigl[V(R,O\cup J)\mid U,O,i,J\bigr]\right\},
}\]
{with $V(\emptyset,O)=0$ and optimal value $V(N,\emptyset)$; this expectation uses the specified independent survival probabilities}. Like many dynamic programming problems, this problem has exponentially many states and each state has $O(n)$ actions, so  {direct solution may be impractical}; moreover, the problem is NP-hard.
\begin{proposition}\label{prop:NP-hard}
     The DSRSE model is NP-hard.
\end{proposition}
\begin{proof}
    Suppose there are $n+1$ tasks and $2n$ slots.
    Let $0$ be a task occupying {slots} $1,\ldots,n$ with probability $1$, and task $i$ taking {slot} $i$ with probability $p_i$ and $n+i$ with probability $(1-p_i)$.
    {Task $0$ is an interval task with deterministic endpoints; each task $i\geq1$ is a singleton task with $\Pr(\widetilde s_i=i)=p_i$ and $\Pr(\widetilde s_i=n+i)=1-p_i$.}
    Then this problem is equivalent to the dynamic maximum stable set problem over a star graph, where $0$ is the center;
    this problem is NP-hard~\cite{DNP}.
\end{proof}
\begin{example}
\label{eg:1}
Consider an instance with $N=[2]$, $M=[3]$, and unit weights, where $D_1^s=\mathbbm{1}_{\{1\}}$, $D_1^e=\mathbbm{1}_{\{2\}}$, $D_2^s=U(\{2,3\})$, $D_2^e=\mathbbm{1}_{\{3\}}$. If the scheduler commits to task $1$ first, then task $1$ occupies $[1,2]$. If task $1$ is scheduled, task $2$ has a probability of $1/2$ of conflicting and $1/2$ of remaining. If it remains, task $2$'s distributions are updated to condition on the fact that it does not conflict with task 1: $D_2^s=\mathbbm{1}_{\{3\}},D_2^e=\mathbbm{1}_{\{3\}}$.
On the other hand, if task $2$ is scheduled first and its start time {is revealed to be} $2$, task $1$ is deleted; if {its start time is revealed to be} $3$, task $1$ remains.
\end{example}

When {every task's interval $I_i$ is deterministic}, DSRSE reduces to the maximum independent set problem of the corresponding interval graph.
As discussed above, it can be represented by~\eqref{LP-P} and~\eqref{LP-D}, so the deterministic problem can be solved efficiently via these LP relaxations.
For an interval graph, maximal cliques {correspond to} elements of $M$, so the constraint $\sum_{i\in C}x_i\leq 1,\forall C\in\mathcal{C}(G)$ in~\eqref{LP-P} simply enforces that each slot is occupied by at most one task.
We consider analogous constraints for DSRSE.

\subsection{Relaxations}

{For interval tasks, let $P_{ik}^s,P_{ik}^e$ be the probabilities of starting and ending at slot $k$. For all tasks, $P_{ik}^o:=\Pr(k\in I_i)$; specifically, for a singleton task, $P_{ik}^o=\Pr(\widetilde s_i=k)$, with $P_{ik}^o=0$ outside $S_i$.}
Let $T=[\min\{m,n\}]$ be the stages of the dynamic program.
For any policy, let $x_{ik}^{ts},x_{ik}^{te},x_{ik}^{to}$ be the probabilities that {an interval task }$i$ is scheduled and starts at, ends at, or occupies slot $k$ at stage $t$, respectively.
Let $x_i^t = \sum_{k=1}^m x_{ik}^{ts}=\sum_{k=1}^{m} x_{ik}^{te}$ be the probability of scheduling $i$ at stage $t$.
{Similarly, for a singleton task, let $x_{ik}^t$ be the probability that $i$ is scheduled at stage $t$ and occupies slot $k$. Set $x_{ik}^t=0$ for $k\notin S_i$ and define $x_i^t:=\sum_{k\in S_i}x_{ik}^t$.}
Analogous to the deterministic case, we require that the probability a slot is occupied is at most one.
We relax DSRSE to the following constraints, together with the occupation identities below:
\begin{align*}
    \max~&\sum_t\sum_{i\in N_{\mathrm{int}}}w_i x_i^t
             +\sum_t\sum_{i\in N_{\mathrm{pt}}}\sum_{k\in S_i}w_i x_{ik}^t\\
    &x_i^t=\sum_{k=1}^m x_{ik}^{ts}=\sum_{k=1}^m x_{ik}^{te},\quad i\in N_{\mathrm{int}},\ t\in T\\
    &x_{ik}^{ts}=0\ (k\notin S_i),\quad x_{ik}^{te}=0\ (k\notin E_i),\quad i\in N_{\mathrm{int}},\ t\in T\\
    &\sum_t\left(\sum_{i\in N_{\mathrm{int}}}x_{ik}^{to}+\sum_{i\in N_{\mathrm{pt}}}x_{ik}^t\right)\leq1,\quad k\in M\\
    &\sum_t x_{ik}^{ts}\leq P_{ik}^s,\quad\sum_t x_{ik}^{te}\leq P_{ik}^e,\quad i\in N_{\mathrm{int}},\ k\in M\\
    &\sum_t x_{ik}^t\leq P^o_{ik},\quad i\in N_{\mathrm{pt}},\ k\in S_i\\
    &x_i^t,x_{ik}^{ts},x_{ik}^{te},x_{ik}^{to},x_{ik}^t\geq0.
\end{align*}
Here $x_{ik}^s:=\sum_t x_{ik}^{ts}\leq P_{ik}^s$ indicates that the probability of $i$ being scheduled and starting at $k$ is at most the marginal probability $P_{ik}^s$, and similarly for $x_{ik}^{e}:=\sum_t x_{ik}^{te}\leq P_{ik}^e$ {and $\sum_t x_{ik}^t\leq P^o_{ik}$}.
{For $i\in N_{\mathrm{int}}$, define $a_i:=\min S_i$, $b_i:=\max S_i$, $c_i:=\min E_i$, and $d_i:=\max E_i$. The set of slots that task $i$ can occupy is $[a_i,d_i]$.}
For each $k\in M$, we can rewrite the slot-capacity constraint as
\begin{align}
&\sum_t\left(\sum_{i\in N_{\mathrm{int}}}x_{ik}^{to}+\sum_{i\in N_{\mathrm{pt}}}x_{ik}^t\right)\notag\\
={}&\sum_t\biggl(\sum_{\substack{i\in N_{\mathrm{int}}\\a_i\leq k<b_i}}x_{ik}^{to}
 +\sum_{\substack{i\in N_{\mathrm{int}}\\c_i<k\leq d_i}}x_{ik}^{to}
 +\sum_{\substack{i\in N_{\mathrm{int}}\\b_i\leq k\leq c_i}}x_{ik}^{to}
 +\sum_{i\in N_{\mathrm{pt}}}x_{ik}^t\biggr)\notag\\
={}&\sum_t\biggl(\sum_{\substack{i\in N_{\mathrm{int}}\\a_i\leq k<b_i}}\sum_{\ell=a_i}^k x_{i\ell}^{ts}
 +\sum_{\substack{i\in N_{\mathrm{int}}\\c_i<k\leq d_i}}\sum_{\ell=k}^{d_i}x_{i\ell}^{te}
 +\sum_{\substack{i\in N_{\mathrm{int}}\\b_i\leq k\leq c_i}}x_i^t
 +\sum_{i\in N_{\mathrm{pt}}}x_{ik}^t\biggr)\leq1.\label{eq:1}
\end{align}
{The first equality partitions the possible occupied slots of each interval task into the disjoint ranges $[a_i,b_i)$, $[b_i,c_i]$, and $(c_i,d_i]$.}
For the second equality, by definition, any $k\in [b_i,c_i]$ is always occupied by $i$ if it is scheduled, so $x_{ik}^{to}=x_i^t$; for $k\in[a_i,b_i)$, the probability of $i$ being scheduled and occupying $k$ at $t$ is equivalent to the probability of $i$ being scheduled and starting at or before $k$ at $t$, and similarly for the ending slots.
Substituting $\sum_{\ell=a_i}^k x_{i\ell}^{ts}=x_i^t-\sum_{\ell=k+1}^{b_i}x_{i\ell}^{ts}$, and the analogous identity for ending slots, gives
\begin{equation}\label{eq:2}
\begin{split}
\sum_t\biggl(&\sum_{\substack{i\in N_{\mathrm{int}}\\a_i\leq k\leq d_i}}x_i^t
 -\sum_{\substack{i\in N_{\mathrm{int}}\\a_i\leq k\leq b_i}}\sum_{\ell=k+1}^{b_i}x_{i\ell}^{ts}-\sum_{\substack{i\in N_{\mathrm{int}}\\c_i\leq k\leq d_i}}\sum_{\ell=c_i}^{k-1}x_{i\ell}^{te}
 +\sum_{i\in N_{\mathrm{pt}}}x_{ik}^t\biggr)\leq1,\quad k\in M.
\end{split}
\end{equation}
Then we have the following primal-dual pair:
{
\begin{align}
\label{DLP-P}\tag{DLP-P}
\begin{split}
\max_{x\geq0}\quad
 &\sum_t\sum_{i\in N_{\mathrm{int}}}w_i x_i^t+\sum_t\sum_{i\in N_{\mathrm{pt}}}\sum_{k\in S_i}w_i x_{ik}^t\\
\text{s.t.}\quad
 &x_i^t=\sum_{k\in S_i}x_{ik}^{ts}=\sum_{k\in E_i}x_{ik}^{te},\quad i\in N_{\mathrm{int}},\ t\in T,\\
 &\eqref{eq:2},\\
 &\sum_t x_{ik}^{ts}\leq P_{ik}^s,\quad i\in N_{\mathrm{int}},\ k\in S_i,\\
 &\sum_t x_{ik}^{te}\leq P_{ik}^e,\quad i\in N_{\mathrm{int}},\ k\in E_i,\\
 &\sum_t x_{ik}^t\leq P^o_{ik},\quad i\in N_{\mathrm{pt}},\ k\in S_i.
\end{split}
\end{align}
(Variables outside their respective supports are zero.)
\begin{align}
\label{DLP-D}\tag{DLP-D}
\begin{split}
\min_{\mu,\nu,\eta\geq0,y,z}\quad
 &\sum_{r=1}^m\mu_r+\sum_{i\in N_{\mathrm{int}}}\sum_{k\in S_i}P_{ik}^s\nu_{ik}^s
 +\sum_{i\in N_{\mathrm{int}}}\sum_{k\in E_i}P_{ik}^e\nu_{ik}^e+\sum_{i\in N_{\mathrm{pt}}}\sum_{k\in S_i}P_{ik}^o\eta_{ik}\\
\text{s.t.}\quad
 &y_{it}+z_{it}+\sum_{r=a_i}^{d_i}\mu_r\geq w_i,\quad i\in N_{\mathrm{int}},\ t\in T,\\
 &-y_{it}-\sum_{r=a_i}^{k-1}\mu_r+\nu_{ik}^s\geq0,\quad i\in N_{\mathrm{int}},\ t\in T,\ k\in S_i,\\
 &-z_{it}-\sum_{r=k+1}^{d_i}\mu_r+\nu_{ik}^e\geq0,\quad i\in N_{\mathrm{int}},\ t\in T,\ k\in E_i,\\
 &\mu_k+\eta_{ik}\geq w_i,\quad i\in N_{\mathrm{pt}},\ k\in S_i.
\end{split}
\end{align}}
{Here $\nu$ corresponds to the endpoint marginal bounds, $y,z$ to the two endpoint-mass equalities, and $\mu$ to the slot-capacity constraints; $\eta$ corresponds to the singleton marginal bounds. The variables $y,z$ are free.}

{For $u_+:=\max\{u,0\}$, define
\begin{align*}
\Delta_{\mathrm{int}}(\mu)
 &:={}\sum_{i\in N_{\mathrm{int}}}\left(\sum_{k\in S_i}P_{ik}^s\sum_{r=a_i}^{k-1}\mu_r
 +\sum_{k\in E_i}P_{ik}^e\sum_{r=k+1}^{d_i}\mu_r\right),\\
R_{\mathrm{pt}}^D(\mu)
 &:={}\sum_{i\in N_{\mathrm{pt}}}\sum_{k\in S_i}P_{ik}^o(w_i-\mu_k)_+.
\end{align*}
The first term sums, over interval tasks $i$, the expected dual price of slots in $[a_i,d_i]$ that task $i$ does not occupy. The second is the singleton contribution remaining after slot prices are applied.

Eliminating $\nu,\eta$ from~\eqref{DLP-D} with the two terms above and having $y=z=0$ gives
\begin{align}
\tag{LP-Dpes}\label{LP-Dpes}
\begin{split}
\min_{\mu\geq0}\quad &\sum_{r=1}^m\mu_r+\Delta_{\mathrm{int}}(\mu)+R_{\mathrm{pt}}^D(\mu)\\
\text{s.t.}\quad &\sum_{r=a_i}^{d_i}\mu_r\geq w_i,\quad i\in N_{\mathrm{int}}.
\end{split}
\end{align}}
Consider the pessimistic interval graph $G_{\pes}$, {whose vertices are only the interval tasks $i\in N_{\mathrm{int}}$, each represented by the interval $[a_i,d_i]$; singleton tasks are omitted.} 
{For the corresponding LP~\eqref{LP-D} defined on $G_{\pes}$,} let $\mu^{\operatorname{LP-D}}_{\pes}$ be an optimal solution, and $\alpha_{\pes}:=\alpha(G_{\pes})=\sum_r[\mu^{\operatorname{LP-D}}_{\pes}]_r$ its optimal value. {If $N_{\mathrm{int}}=\emptyset$, set $\alpha_{\pes}=0$ and $\mu^{\operatorname{LP-D}}_{\pes}=0$.}
{Extending $\mu=\mu^{\operatorname{LP-D}}_{\pes}$ and $y=z=0$ by the eliminated values of $\nu,\eta$ yields a feasible solution of~\eqref{DLP-D}. Hence its optimal value $\alpha^*$ is at most
\[
\alpha_{\pes}+\Delta_{\mathrm{int}}(\mu^{\operatorname{LP-D}}_{\pes})+R_{\mathrm{pt}}^D(\mu^{\operatorname{LP-D}}_{\pes}).
\]
The interval contribution can also be written as
\begin{align*}
\Delta_{\mathrm{int}}(\mu)
 =\sum_r\mu_r\biggl(&\sum_{\substack{i\in N_{\mathrm{int}}\\a_i\leq r\leq b_i}}(1-P_{ir}^o)
 +\sum_{\substack{i\in N_{\mathrm{int}}\\c_i\leq r\leq d_i}}(1-P_{ir}^o)\biggr).
\end{align*}
Indeed, on the start band $1-P_{ir}^o=\sum_{k\in S_i:k>r}P_{ik}^s$, and on the end band $1-P_{ir}^o=\sum_{k\in E_i:k<r}P_{ik}^e$; the slot ranges include gaps in endpoint supports. Every independent set of pessimistic intervals can be committed first without conflict, so
\[
\alpha_{\pes}\leq\operatorname{OPT}_{\operatorname{DSRSE}}\leq\alpha^*
\leq\alpha_{\pes}+\Delta_{\mathrm{int}}(\mu^{\operatorname{LP-D}}_{\pes})+R_{\mathrm{pt}}^D(\mu^{\operatorname{LP-D}}_{\pes}).
\]
The singleton term satisfies $R_{\mathrm{pt}}^D(\mu^{\operatorname{LP-D}}_{\pes})\leq\sum_{i\in N_{\mathrm{pt}}}w_i$, and may be strictly smaller. When $N_{\mathrm{pt}}=\emptyset$ and the endpoints are deterministic, $\Delta_{\mathrm{int}}(\mu^{\operatorname{LP-D}}_{\pes})=0$ and the bound is tight.}

We conclude this section by considering the special case where {all tasks are singleton tasks, $N=N_{\mathrm{pt}}$}.

\begin{proposition}
    If each task occupies only a single slot, {a policy that always selects an available task of maximum weight is optimal.}
\end{proposition}
\begin{proof}
    We use an adversarial argument.
    Consider an adversary that places tasks $N$ into $[m]$ {slots} in advance according to their {distributions $D_i^s$ }for each $i\in N$.
    Since the length of each task is $1$, if the {slots} of the tasks are known, it is optimal to pick the largest weighted task of each {slot}.
    Without loss of generality, assume the weights of the largest weighted task of each {slot} to be $w_1\geq w_2\geq \ldots w_m\geq 0$.
    Without knowing the {slots}, the greedy policy would pick $w_1,w_2,\ldots, w_m$ which is equivalent to the case where {slots} are known.
    Thus, for any realization, the greedy policy returns the set of non-conflicting tasks with the maximum weights.
\end{proof}

\section{Conservative DSRSE}
\label{section:Conservatice DSRSE}

In this section, we consider a {variant} of DSRSE, where a task is deleted whenever it possibly conflicts with the committed tasks. 
In the original DSRSE model, a remaining unscheduled task's {interval distribution is updated} throughout the horizon depending on the tasks that have already been committed; 
that is, these distributions evolve with each stage of the decision process. 
{Although the distributions of remaining tasks can be updated after each commitment to condition on occupied slots, optimizing these decisions requires accounting for exponentially many possible states.}
With this motivation, we consider a conservative case in which the {interval distributions} are not updated after each decision. 
This simplification removes the stage dependence of the distributions and yields a simpler dynamic model.

As in DSRSE, we have $M,N,w,D${\unskip, with the same interval and singleton task specifications and independent interval distributions $D=(Q_i)_{i\in N}$. 
Once the slots $I_i$ occupied by an accepted task are revealed}, we delete any remaining task $j$ if  {$I_i\cap H_j \neq\emptyset$, where $H_j:=\bigcup_{J:Q_j(J)>0}J$ is the set of slots that $j$ can possibly occupy}; 
otherwise, we keep $j$.  {Thus, this rule uses only potential conflicts and does not query $j$'s actual interval. A surviving task retains its original distribution $Q_j$, since survival depends only on its support and other tasks' revealed intervals. 
As before, a feasible solution is a policy choosing the next available task before its interval is disclosed, with the objective of maximizing expected total accepted weight}.
We call this problem \emph{conservative dynamic scheduling with random start and end times (CDSRSE)}.

\begin{example}
\label{eg:2}
Consider an instance with $N=[2]$, $M=[3]$, and unit weights, where $D_1^s=\mathbbm{1}_{\{1\}}$, $D_1^e=\mathbbm{1}_{\{2\}}$, $D_2^s=U(\{2,3\})$, $D_2^e=\mathbbm{1}_{\{3\}}$. 
If we commit to task $1$ first, then task $1$ occupies $[1,2]$.
For CDSRSE, task $2$ is always deleted, since {slot} $2$ is in its support, which is now occupied by task $1$.
\end{example}

\begin{theorem}
     The CDSRSE model is also NP-hard.
\end{theorem}
\begin{proof}
    As in the previous proof, consider the instance with $n+1$ tasks and $2n$ slots. 
    Task $0$ occupies slots $1,\ldots,n$ with probability $1$, and task $i$ occupies slot $i$ with probability $p_i$ and $n+i$ with probability $q_i:=1-p_i$. {As in the proof of Proposition \ref{prop:NP-hard}, task $0$ is an interval task and tasks $1,\ldots,n$ are singleton tasks.}
    Following \cite{DNP}, any solution for this problem is determined by when to select task $0$, and what is selected before trying to select task $0$.
    Then the problem can be formulated as
    \begin{equation}
    \label{Conserv-JS}
    \tag{Conserv-JS}
    \max_{S\subseteq N}\left\{\sum_{i\in S}w_i+(1-\prod_{i\in S}q_i)\sum_{j\notin S}w_j+\prod_{i\in S}q_i{w_0}\right\}.
    \end{equation}
    By rearranging, we get 
    \[
    \max_{S\subseteq N}\left\{\sum_{i\in N}w_i+\prod_{i\in S}q_i(w_0-\sum_{j\notin S}w_j)\right\}.
    \]
    {For the positive weights and probabilities used below, we restrict the following logarithmic formulation to subsets with $w_0-\sum_{j\notin S}w_j>0$; choosing $S=N$ gives a positive contribution, so this restriction does not change the optimum.}
    By dropping constant terms and applying the natural logarithm, we get 
    \[\max_{S\subseteq N}\left\{\sum_{i\in S}\ln q_i+\ln(w_0-\sum_{j\notin S}w_j)\right\}.\]
    Given a set of positive numbers $\zeta_i$, $i \in N$, the partitioning problem asks for a subset $S \subseteq N$ such that $\sum_{i\in S} \zeta_i = \sum_{j\notin S} \zeta_j$; without loss of generality, we assume that $\sum_{i\in N} \zeta_i = 2$.
    By setting parameters of the problem above as $q_i = e^{- \zeta_i}$, $p_i = 1-e^{-\zeta_i}$, $w_i =\zeta_i$,and $w_0 = \sum_{i\in N} \zeta_i$, we formulate
    \[\max_{{\emptyset\ne S\subseteq N}}\left\{-\sum_{i\in S} \zeta_i+\ln(\sum_{i \in S} \zeta_i)\right\}.\]
Proceeding similarly to~\cite{Star-Graph-NP-hard}, there is a set $S \subseteq N$ with $\sum_{i\in S} \zeta_i=1$ if and only if the optimum of this problem is $-1$.
\end{proof}

{The sets of slots that tasks can occupy are $H_i=[a_i,d_i]$ for $i\in N_{\mathrm{int}}$ and $H_i=S_i$ for $i\in N_{\mathrm{pt}}$.}
Unlike DSRSE, in CDSRSE the distributions are not updated for remaining tasks.
{Whether a task is accepted is independent of its own realization: before acceptance, the policy and deletion rule use only its known support and other tasks' observations. Thus $x_{ik}^{ts}=P_{ik}^sx_i^t$ and $x_{ik}^{te}=P_{ik}^ex_i^t$ for interval tasks, and $x_{ik}^t=P_{ik}^ox_i^t$ for singleton tasks, where $x_i^t$ is the probability of acceptance at stage $t$. Writing $x_i=\sum_t x_i^t$ for every task,} we can express the constraint $\Pr[\text{$k$ occupied by any $i$}]\leq1$ as $\sum_{i=1}^nP_{ik}^ox_i\leq1$.
Consider the LP relaxation and its dual,
{
\begin{align}
\label{CDLP-P}\tag{CDLP-P}
\begin{split}
\max_{x\geq0}\quad &\sum_{i\in N}w_ix_i\\
\text{s.t.}\quad &\sum_{i\in N}P_{ik}^ox_i\leq1,\quad k\in M,\\
&x_i\leq1,\quad i\in N_{\mathrm{pt}},
\end{split}
\end{align}
\begin{align}
\label{CDLP-D}\tag{CDLP-D}
\begin{split}
\min_{\mu,\lambda\geq0}\quad &\sum_{r=1}^m\mu_r+\sum_{i\in N_{\mathrm{pt}}}\lambda_i\\
\text{s.t.}\quad &\sum_{r=1}^mP_{ir}^o\mu_r\geq w_i,\quad i\in N_{\mathrm{int}},\\
&\sum_{r=1}^m P_{ir}^o\mu_r+\lambda_i\geq w_i,\quad i\in N_{\mathrm{pt}}.
\end{split}
\end{align}
The bounds $x_i\leq1$ are redundant for interval tasks, since every such task always occupies $[b_i,c_i]$. They must be imposed explicitly for singleton tasks.}

Let $p^*:=\min\{P_{ir}^o:P_{ir}^o>0,\ i\in N,\ r\in M\}$.
As above, consider the pessimistic interval graph $G_{\pes}$ {on $N_{\mathrm{int}}$ }and the corresponding optimal solution $\mu^{\operatorname{LP-D}}_{\pes}$ of~\eqref{LP-D} defined on $G_{\pes}$.
Let $\hat\mu:=\mu^{\operatorname{LP-D}}_{\pes}/p^*$; then for each {$i\in N_{\mathrm{int}}$},
\[
\sum_rP_{ir}^o\hat\mu_r\geq\sum_{r\in[a_i,d_i]}[\mu^{\operatorname{LP-D}}_{\pes}]_r\geq w_i.
\]
{Define
\[
R_{\mathrm{pt}}^C(\mu):=\sum_{i\in N_{\mathrm{pt}}}\left(w_i-\sum_r P_{ir}^o\mu_r\right)_+.
\]
Choosing $\hat{\lambda}_i=(w_i-\sum_r P_{ir}^o\hat\mu_r)_+$ makes $(\hat\mu,\hat\lambda)$ feasible for~\eqref{CDLP-D}. Consequently,
\[
\alpha_{\pes}\leq\operatorname{OPT}_{\operatorname{CDSRSE}}
\leq\operatorname{val}(\mathrm{CDLP\text{-}P})
\leq\frac{\alpha_{\pes}}{p^*}+R_{\mathrm{pt}}^C(\mu^{\operatorname{LP-D}}_{\pes}/p^*).
\]}
However, this bound is weak when $p^*$ is small.

\begin{remark}
    We emphasize that $\sum_{i=1}^{n}P_{ik}^{o}x_i\leq 1$ is not a valid inequality for the original DSRSE, because the distribution of start and end times of $i$ depends on the stage at which it is committed. 
Consider Example~\ref{eg:1} and a policy that always schedules task $1$ first and then task $2$ if possible.
Under this policy, the probability that slot $2$ is occupied is $1\cdot x_1 = 1$, but $P_{12}^ox_1+P_{22}^o x_2= 1\cdot x_1+\frac{1}{2}x_2=1+\frac{1}{4}>1$.
For Example~\ref{eg:2}, after task $1$ is scheduled, the probability of scheduling task $2$ is $0$, so $P_{12}^ox_1+P_{22}^o x_2= 1$.
\end{remark}
{Every independent set of the pessimistic interval graph, which contains only interval tasks, can be committed first in any order without conflict under either model. Hence $\alpha_{\pes}$ is a lower bound for the optimal value of CDSRSE. DSRSE can follow a CDSRSE policy, postponing additional survivors until that policy ends, so its optimal value is at least that of CDSRSE.} Therefore, for DSRSE and CDSRSE, if $\operatorname{OPT}_{\operatorname{DSRSE}}$ and $\operatorname{OPT}_{\operatorname{CDSRSE}}$ denote their respective optimal values, then
{
\begin{align*}
\alpha_{\pes}
\leq\operatorname{OPT}_{\operatorname{CDSRSE}}
\leq\operatorname{OPT}_{\operatorname{DSRSE}}\leq\alpha_{\pes}+\Delta_{\mathrm{int}}(\mu^{\operatorname{LP-D}}_{\pes})+R_{\mathrm{pt}}^D(\mu^{\operatorname{LP-D}}_{\pes}).
\end{align*}}
Thus the gap between the two optimal values is bounded by
{
\[
\operatorname{OPT}_{\operatorname{DSRSE}}-\operatorname{OPT}_{\operatorname{CDSRSE}}
\leq\Delta_{\mathrm{int}}(\mu^{\operatorname{LP-D}}_{\pes})+R_{\mathrm{pt}}^D(\mu^{\operatorname{LP-D}}_{\pes}).
\]}

However, this bound can be loose. For example, suppose that {$N_{\mathrm{pt}}=\emptyset$ and }$ S_i \cup E_i = M $ for each task $i$; in other words, a small amount of probability mass for the task's start and/or end time is spread over the entire set $M = [m]$. Then every {slot} $r\in[m]$ belongs to the support of each task, and the terms $1-P_{ir}^o$ can be close to one for many pairs $(i,r)$. In this case, the bound above can be close to
\(
n\sum_r [{\mu^{\operatorname{LP-D}}_{\pes}}]_r .
\)
The bound is more informative when the start and end time distributions are sufficiently localized.

\section{Uniform Weights}

In this section, we consider the case where task weights are uniform, $ w_i = 1 $ for $ i \in N$. 
Let $\alpha_{\pes}$ be the optimal value of the pessimistic interval graph.
{For $G_{\pes}$, because of the cardinality objective there is a minimum clique cover represented by a set of slots $\hat{J}\subseteq M$ with $|\hat{J}|=\alpha_{\pes}$ such that $[a_i,d_i]\cap \hat{J}\neq\emptyset$ for every $i\in N_{\mathrm{int}}$; this clique cover can be derived from~\eqref{LP-D}.}
{For each $i\in N_{\mathrm{int}}$, choose one covering slot $k_i\in \hat{J}\cap[a_i,d_i]$, and let $C_k:=\{i\in N_{\mathrm{int}}:k_i=k\}$. These sets partition the interval tasks, even though an interval may contain several covering slots in $\hat{J}$. For a singleton task $j$, write $q_j(\hat{J}):=\sum_{k\in \hat{J}}P_{jk}^o$.}
Let $x$ be an optimal solution of~\eqref{DLP-P}, and $\alpha$ be its optimal value;
{write $x_i=\sum_t x_i^t$, $x_{i\ell}^s=\sum_t x_{i\ell}^{ts}$, $x_{i\ell}^e=\sum_t x_{i\ell}^{te}$ for interval tasks, and $x_{jk}=\sum_t x_{jk}^t$ for singleton tasks. Applying the slot-capacity constraint underlying~\eqref{eq:2} to the interval tasks in $C_k$ and all singleton tasks gives, for $k\in \hat{J}$,
\[
\sum_{i\in C_k}x_i+\sum_{j\in N_{\mathrm{pt}}}x_{jk}
\leq1+\sum_{i\in C_k}\left(\sum_{\ell=k+1}^{b_i}x_{i\ell}^s+\sum_{\ell=c_i}^{k-1}x_{i\ell}^e\right).
\]
The omitted interval tasks have nonnegative occupation terms, and empty sums are zero.}

{Summing over $k\in \hat{J}$ accounts for all interval tasks and for singleton tasks accepted at slots in $\hat{J}$. Adding the singleton contributions at slots outside $\hat{J}$ gives}
{
\[\begin{aligned}
\alpha={}&\sum_{i\in N_{\mathrm{int}}}x_i+\sum_{j\in N_{\mathrm{pt}}}\sum_{k\in S_j}x_{jk}\\
\leq{}&\alpha_{\pes}+\sum_{i\in N_{\mathrm{int}}}\left(\sum_{\ell=k_i+1}^{b_i}x_{i\ell}^s+\sum_{\ell=c_i}^{k_i-1}x_{i\ell}^e\right)
 +\sum_{j\in N_{\mathrm{pt}}}\sum_{k\in S_j\setminus \hat{J}}x_{jk}\\
\leq{}&\alpha_{\pes}+\sum_{i\in N_{\mathrm{int}}}(1-P_{ik_i}^o)
 +\sum_{j\in N_{\mathrm{pt}}}(1-q_j(\hat{J})).
\end{aligned}\]}
The last line is from the constraint $\sum_t x_{i\ell}^{ts}\leq P_{i\ell}^s$, and the analogous constraint for $x_{i\ell}^e${, together with $x_{jk}\leq P_{jk}^o$}.

For CDSRSE, {let $x$ be an optimal solution of~\eqref{CDLP-P} and $\alpha$ its value. The same cover groups give
\[
\sum_{i\in N_{\mathrm{int}}}P_{ik_i}^ox_i+\sum_{j\in N_{\mathrm{pt}}}q_j(\hat{J})x_j\leq\alpha_{\pes}.
\]
If $N_{\mathrm{int}}\neq\emptyset$, set $\underline p:=\min_{i\in N_{\mathrm{int}}}P_{ik_i}^o>0$. Using $0\leq x_j\leq1$ for singleton tasks,
\begin{align*}
\underline p\,\alpha
&\leq\alpha_{\pes}+\sum_{j\in N_{\mathrm{pt}}}(\underline p-q_j(\hat{J}))x_j 
\leq\alpha_{\pes}+\sum_{j\in N_{\mathrm{pt}}}(\underline p-q_j(\hat{J}))_+.
\end{align*}
Therefore,
\[
\alpha\leq\frac{\alpha_{\pes}}{\underline p}
 +\sum_{j\in N_{\mathrm{pt}}}\left(1-\frac{q_j(\hat{J})}{\underline p}\right)_+.
\]
If $N_{\mathrm{int}}=\emptyset$, take $\hat{J}=\emptyset$; the DSRSE bound above and the CDSRSE constraints $x_j\leq1$ both give the bound $|N_{\mathrm{pt}}|$.}

\section{Computational Experiments}

In this section, we test the bounds derived above. {All generated tasks use the interval-task specification, so $N_{\mathrm{int}}=N$ and $N_{\mathrm{pt}}=\emptyset$, including any inputs with deterministic equal endpoints. The singleton terms in the relaxations and bounds are therefore zero.}
We use the following $(n,m)$ combinations: $(8,12)$, $(10,15)$, {$(12,18)$, }$(14,21)$, $(16,24)$, $(18,27)$, $(19,29)$, $(20,30)$, $(40,60)$.
For each $(n,m)$, we generate {$10$ base }instances and report average results {separately for the four settings described below, giving $360$ instances in total}.

For each task $i$, we first sample $x,y$ uniformly and independently from $[m]$, and let $a_i=\min\{x,y\}$, $d_i=\max\{x,y\}$.
Similarly, we sample $u,v$ uniformly and independently from $[a_i,d_i]$, and set $b_i=\min\{u,v\}$, $c_i=\max\{u,v\}$.
The start and end times follow uniform distributions on $[a_i,b_i]$, $[c_i,d_i]$, respectively. Finally, we sample task weights uniformly from $[0,1]$.
{Conditional on these support parameters, the two endpoints are independent. Their product distribution is the interval-task law $Q_i$, and $a_i\leq b_i\leq c_i\leq d_i$ holds by construction.}

We conducted all of our computational experiments on a laptop computer with an Apple M1 Pro CPU and 16 GB of RAM {using Julia 1.11.3}, solving LPs using HiGHS~\cite{HiGHS} with {feasibility tolerances $10^{-7}$ and interior-point optimality tolerance $10^{-8}$}.

Since computing the optimal values of DSRSE and {CDSRSE }is impractical, we estimate the expected maximum-weight schedule $\E[\alpha(G)]$ by repeatedly sampling the start and end times of each task and solving the deterministic maximum weight independent set $1,000$ times; {the exact expectation }is an upper bound for the optimal values of DSRSE and CDSRSE{\unskip, and the reported value is its Monte Carlo estimate}.
We compare it with the {computed optimal }values of~\eqref{DLP-P},~\eqref{CDLP-P}, and with $\alpha_{\pes}$.
{For the separately reported values of~\eqref{DLP-P} and~\eqref{CDLP-P}, LP coefficients $P_{ik}^s,P_{ik}^e$ are known from the input instance and $P_{ik}^o$ is computed exactly accordingly.}
Due to the poor performance of the bounds
\[
  \sum_r[\mu_{\pes}^{\operatorname{LP-D}}]_r\left(
    1+\sum_{i:{a_i\leq r\leq b_i}}(1-P_{ir}^o)
  +\sum_{i:{c_i\leq r\leq d_i}}(1-P_{ir}^o)\right)
  \quad\text{and}\quad\frac{\alpha_{\pes}}{p^*},
\]
we do not include them in the plots.

On the primal side, we consider the heuristic that picks the remaining task with the largest ratio $w_i/\E[\text{length of }i]=w_i/\sum_r P_{ir}^o$, which can be interpreted as the weight gained per expected occupied slot; we call this heuristic \emph{(C)DSRSE ratio}.
{At each stage $t$, the expected length of a surviving task is recomputed given the updated occupied slots and distribution.}
We also consider the heuristic that, in stage $t$, maximizes {$w_i-\sum_{\text{available }j\neq i}\widehat p_{ij}w_j$. For DSRSE, $\widehat p_{ij}$ is the pairwise overlap probability computed directly based on the current surviving tasks and occupied slots. For CDSRSE, it is the indicator that the intervals $[a_i,d_i],~[a_j,d_j]$ overlap. Thus this score approximates }net gained weight {using a conflict surrogate}; we call this heuristic \emph{(C)DSRSE weight}.
We finally consider the heuristic that repeatedly solves~\eqref{DLP-P},~\eqref{CDLP-P} and picks $i$ with the largest $x_i^1$ and $x_i$ for DSRSE and CDSRSE respectively. 
We use the optimal solution returned by HiGHS; different optimal solutions may yield different scores.
For all policies, ties between maximum scores are broken by the
smallest task index.
We call this heuristic \emph{(C)DSRSE adaptive LP.}

{For each evaluated instance, we reuse the $1,000$ replications used to estimate the expected stability number to evaluate the expected performance of each policy.
For each replication, each task's interval is sampled once and retained throughout; surviving tasks are therefore compatible with all previous commitments;
all policies are tested on the same samples.
For DSRSE, after a task is committed by a policy, its hidden interval is revealed and all conflicting tasks are deleted.
For CDSRSE, a remaining task $j$ is deleted when $[a_j,d_j]$ overlaps the committed interval.
The policy updates the necessary values and repeats until all tasks are either committed or deleted.
In particular, both adaptive LP policies use the same solver
settings at every stage: DSRSE computes the coefficients exactly
from the conditional distributions of the surviving tasks and
resolves the LP; CDSRSE resolves its LP using the original
coefficients of the surviving tasks.
}

We test four sets of instances for each $(n,m)$ pair.
We first generate $n$ {tasks }and treat this as a {short}, dense instance.
Then we consider a sparser instance where we {retain a uniformly sampled subset of $\lfloor n/2\rfloor$ tasks.
For the long setting, we double the spans $b_i-a_i$, $c_i-b_i$, and $d_i-c_i$ by expanding }$[b_i,c_i]$ {about its midpoint and then expanding }$[a_i,b_i]$ and $[c_i,d_i]$ from {the new boundaries $[b_i,c_i]$; translation and horizon extension are applied when necessary. 
The new expected occupied length is twice the original expected length minus one.
For }this case, we also consider both dense and sparse instances of $n$ and {$\lfloor n/2\rfloor$ }tasks, respectively.

{Within each size-density-length class, we first average each policy's estimated expected performance and each bound over the ten instances.
We then normalize these averages by the class average of the simulated expected stability number. 
Reported average gaps are the absolute relative differences calculated within each class, averaged equally over the $36$ classes;
overall, $360$ instances are tested.}

{
In Figure~\ref{fig:DSRSE}, we present the value of each bound and
heuristic after normalization by the estimated expected stability
number. The LP and pessimistic values correspond to~\eqref{DLP-P}
and $\alpha_{\pes}$, respectively. Across the $36$ classes, the
value of~\eqref{DLP-P} is on average $6.52\%$ above the simulated expected stability number,
whereas the pessimistic value is on average $15.91\%$ below it.
The average gaps from the simulated expected stability number are $3.06\%$ for DSRSE-weight
and $5.82\%$ for DSRSE-ratio. DSRSE-weight has the higher class mean
in $34$ of $36$ classes; DSRSE-ratio is higher for the dense-short
and sparse-long classes with base size $(n,m)=(40,60)$.
For the DSRSE adaptive LP heuristic, the average gap from the same
simulated expected stability number is $9.49\%$. Its class mean is below that of DSRSE-weight
in all $36$ classes and below that of DSRSE-ratio in $33$ classes.
Thus, among the three DSRSE heuristics, the weight heuristic has
the smallest average gap in these experiments.

Figure~\ref{fig:CDSRSE} presents analogous results for CDSRSE.
The computed value of~\eqref{CDLP-P} is below the simulated expected stability number in $35$ of $36$ classes, with an average
signed reduction of $3.91\%$. The average gaps from the
simulated expected stability number are $12.60\%$ for CDSRSE-weight, $17.83\%$ for
CDSRSE-ratio, and $9.81\%$ for CDSRSE adaptive LP. CDSRSE-weight
has the higher class mean than CDSRSE-ratio in $32$ of $36$ classes.
The adaptive LP heuristic has a higher class mean than
CDSRSE-weight in $30$ classes and than CDSRSE-ratio in $35$ classes,
and has the smallest average gap among the three CDSRSE heuristics.
All three CDSRSE heuristics have larger average gaps than their
DSRSE counterparts under the more conservative deletion rule.
As we observe in Section~\ref{section:Conservatice DSRSE}, $x_i$ in~\eqref{CDLP-P} is stage independent and represents the overall likelihood of selecting $i$. This intuition may indicate why adaptive LP is a good heuristic in this case.
}

\begin{figure}[H]
  \centering
  \begin{subfigure}[h]{0.49\textwidth}
    \centering
    \includegraphics[scale=0.35]{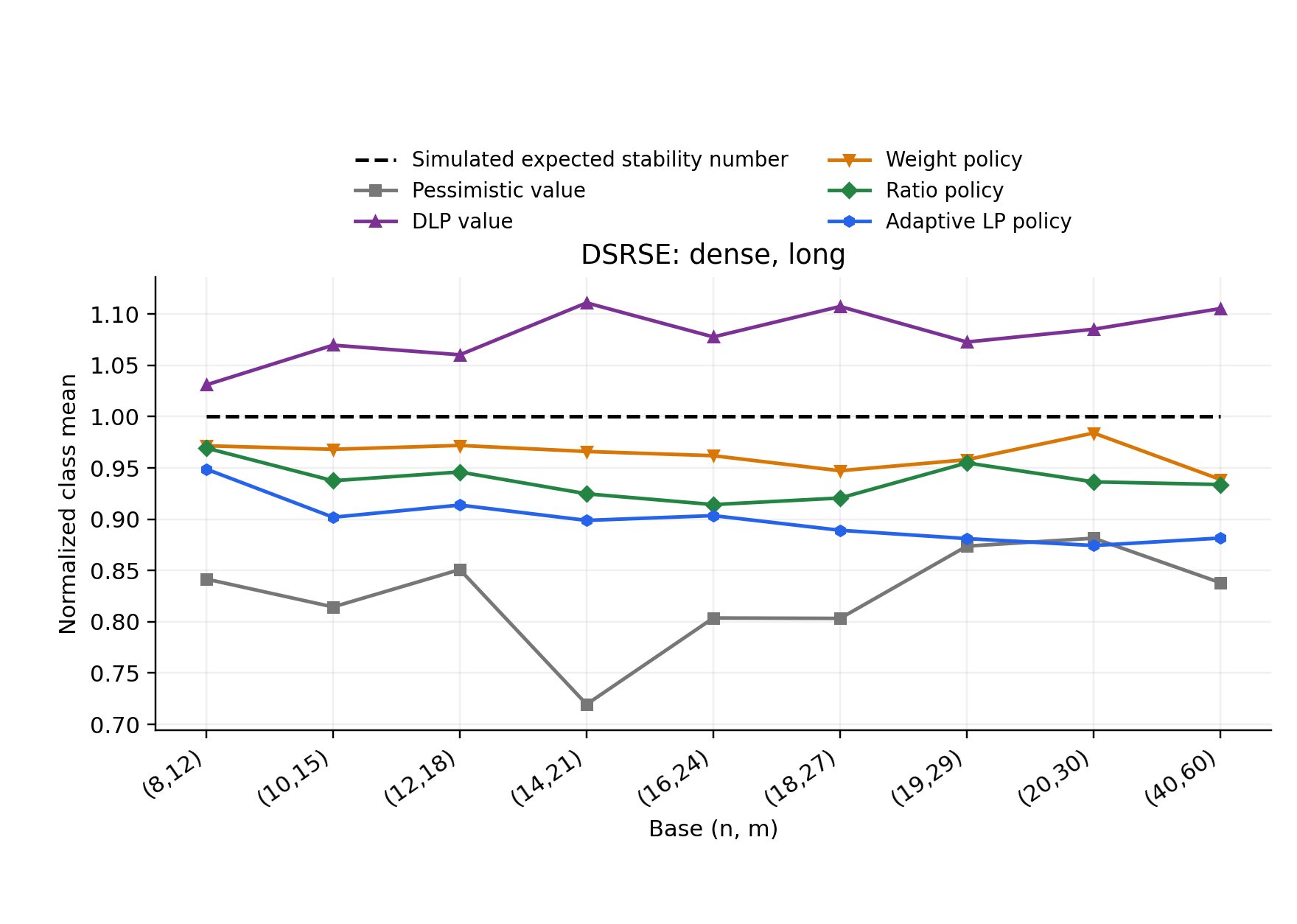}
    \caption{DSRSE: intervals with uniform start and end times and long expected length, densely packed across the slots.}
    \label{fig:DSRSE-Uni}
  \end{subfigure}
  \hfill
  \begin{subfigure}[h]{0.49\textwidth}
    \centering
    \includegraphics[scale=0.35]{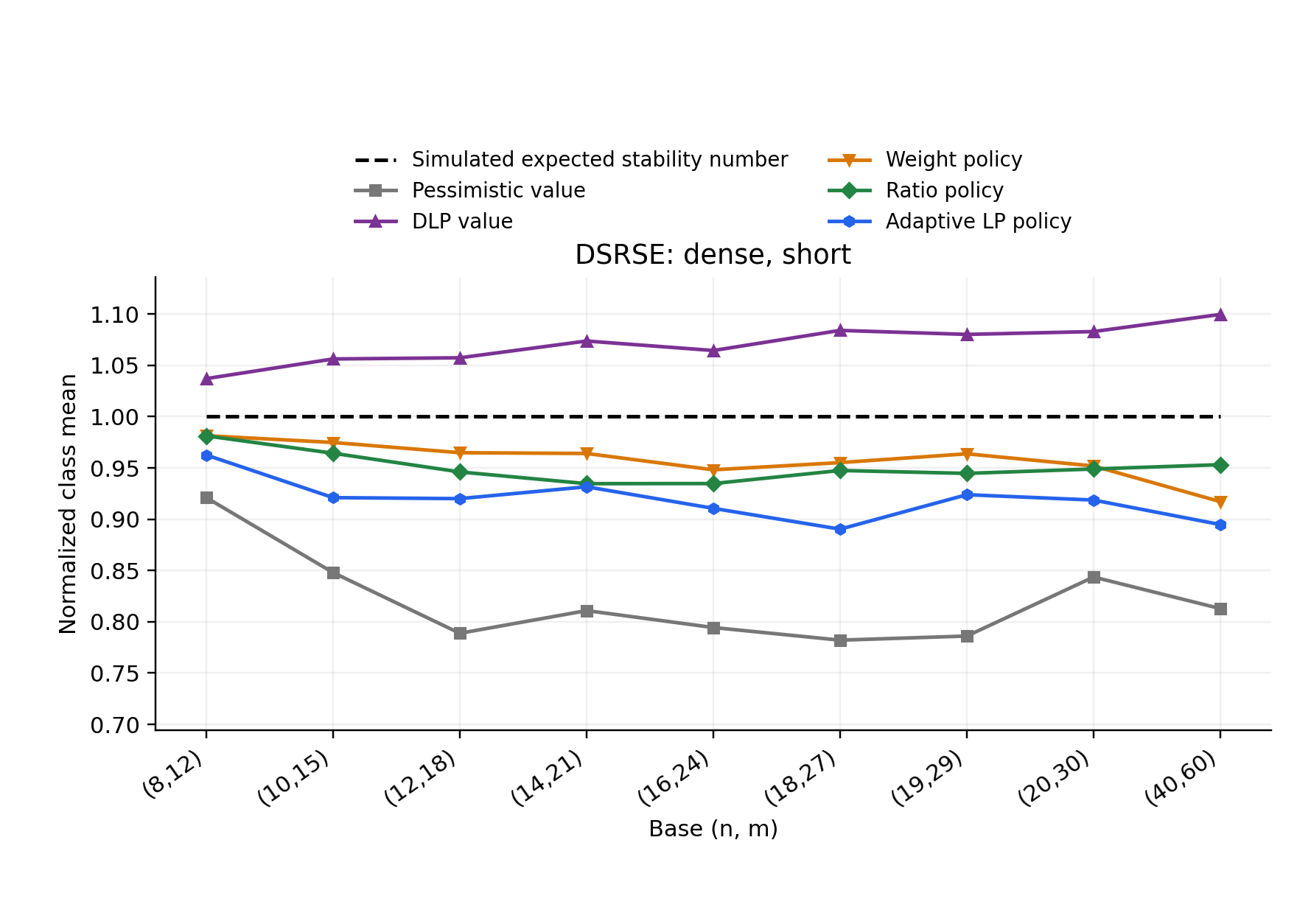}
    \caption{DSRSE: intervals with uniform start and end times and short expected length, densely packed across the slots.}
    \label{fig:DSRSE-Gau}
  \end{subfigure}

  \begin{subfigure}[h]{0.49\textwidth}
    \centering
    \includegraphics[scale=0.35]{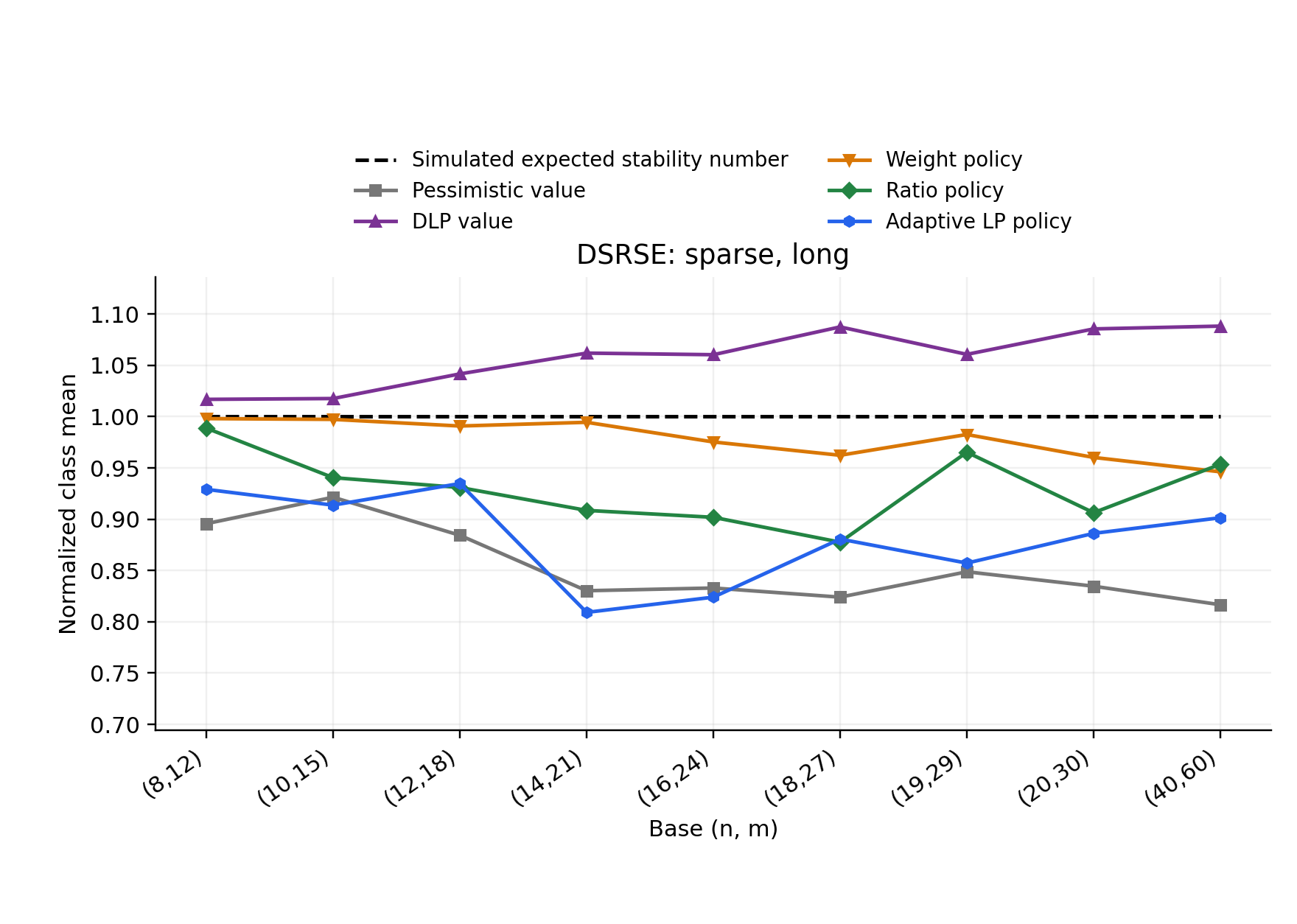}
    \caption{DSRSE: intervals with uniform start and end times and long expected length, sparsely packed across the slots.}
    \label{fig:DSRSE-Uni-sparse}
  \end{subfigure}
  \hfill
  \begin{subfigure}[h]{0.49\textwidth}
    \centering
    \includegraphics[scale=0.35]{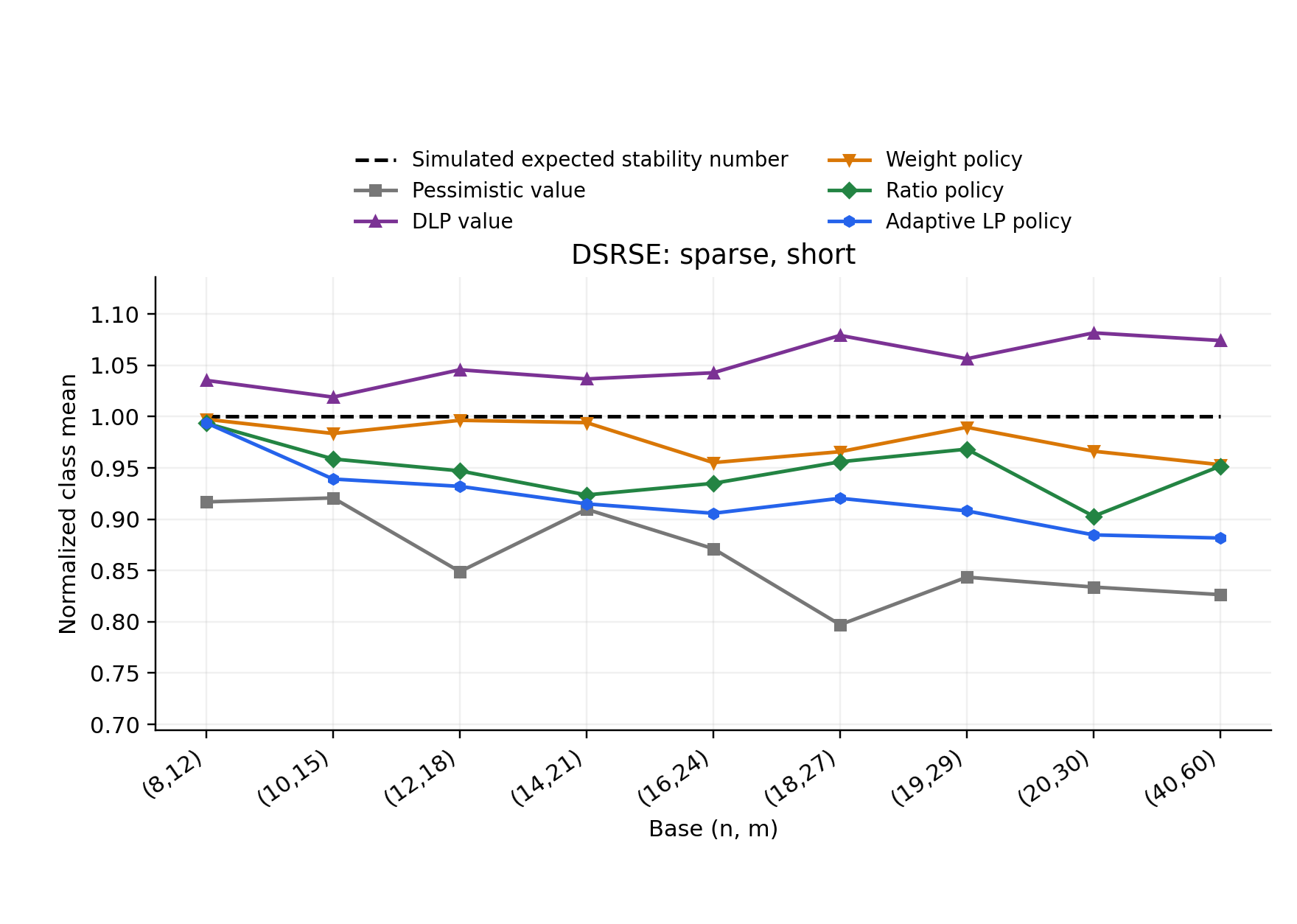}
    \caption{DSRSE: intervals with uniform start and end times and short expected length, sparsely packed across the slots.}
    \label{fig:DSRSE-Gau-sparse}
  \end{subfigure}
  \caption{DSRSE. {All values are class means normalized by the corresponding class mean of the simulated expected stability number. All policies are evaluated on the same realizations.}}
  \label{fig:DSRSE}
\end{figure}

\begin{figure}[H]
  \centering
  \begin{subfigure}[h]{0.49\textwidth}
    \centering
    \includegraphics[scale=0.35]{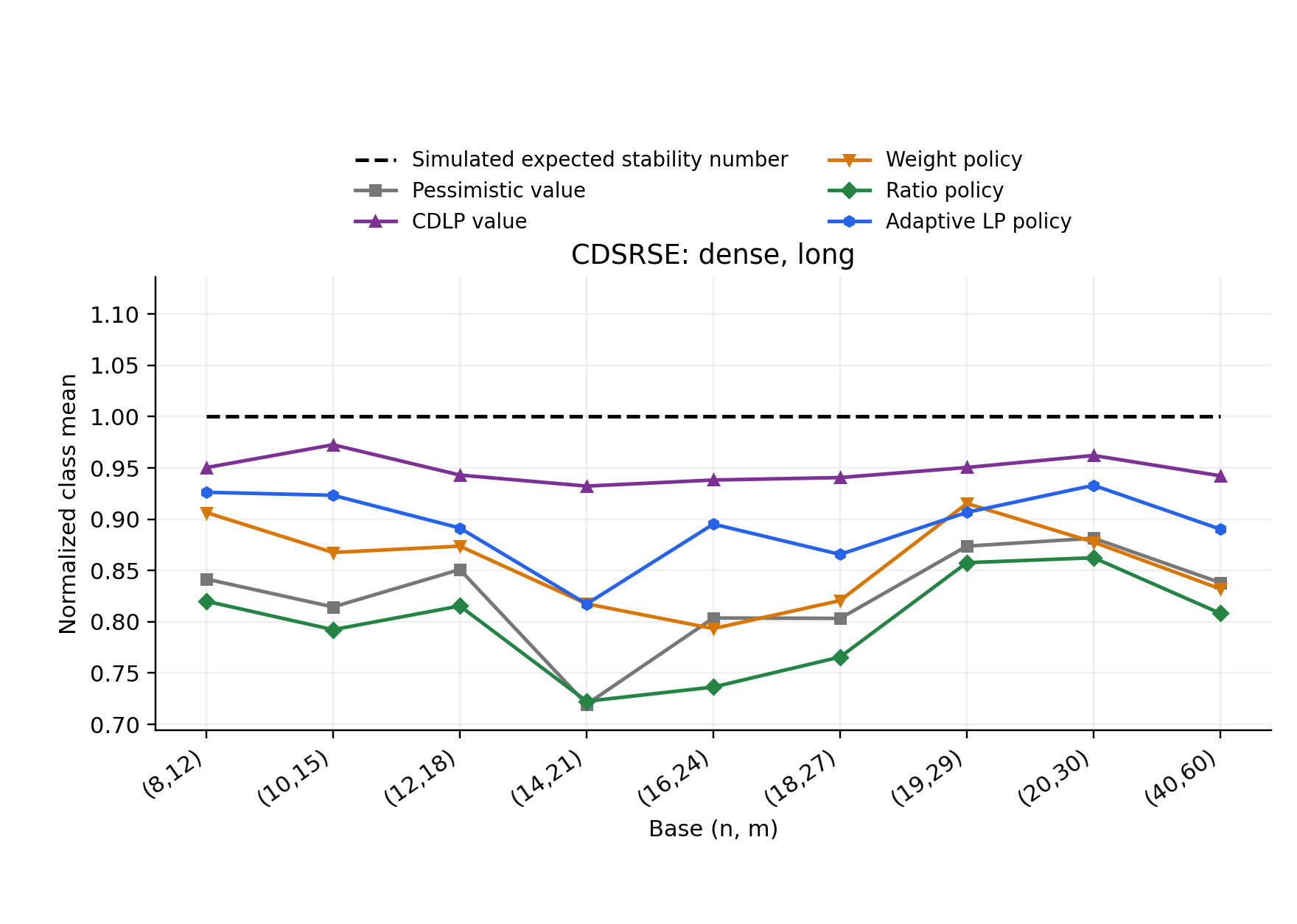}
    \caption{CDSRSE: intervals with uniform start and end times and long expected length, densely packed across the slots.}
    \label{fig:CDSRSE-Uni}
  \end{subfigure}
  \hfill
  \begin{subfigure}[h]{0.49\textwidth}
    \centering
    \includegraphics[scale=0.35]{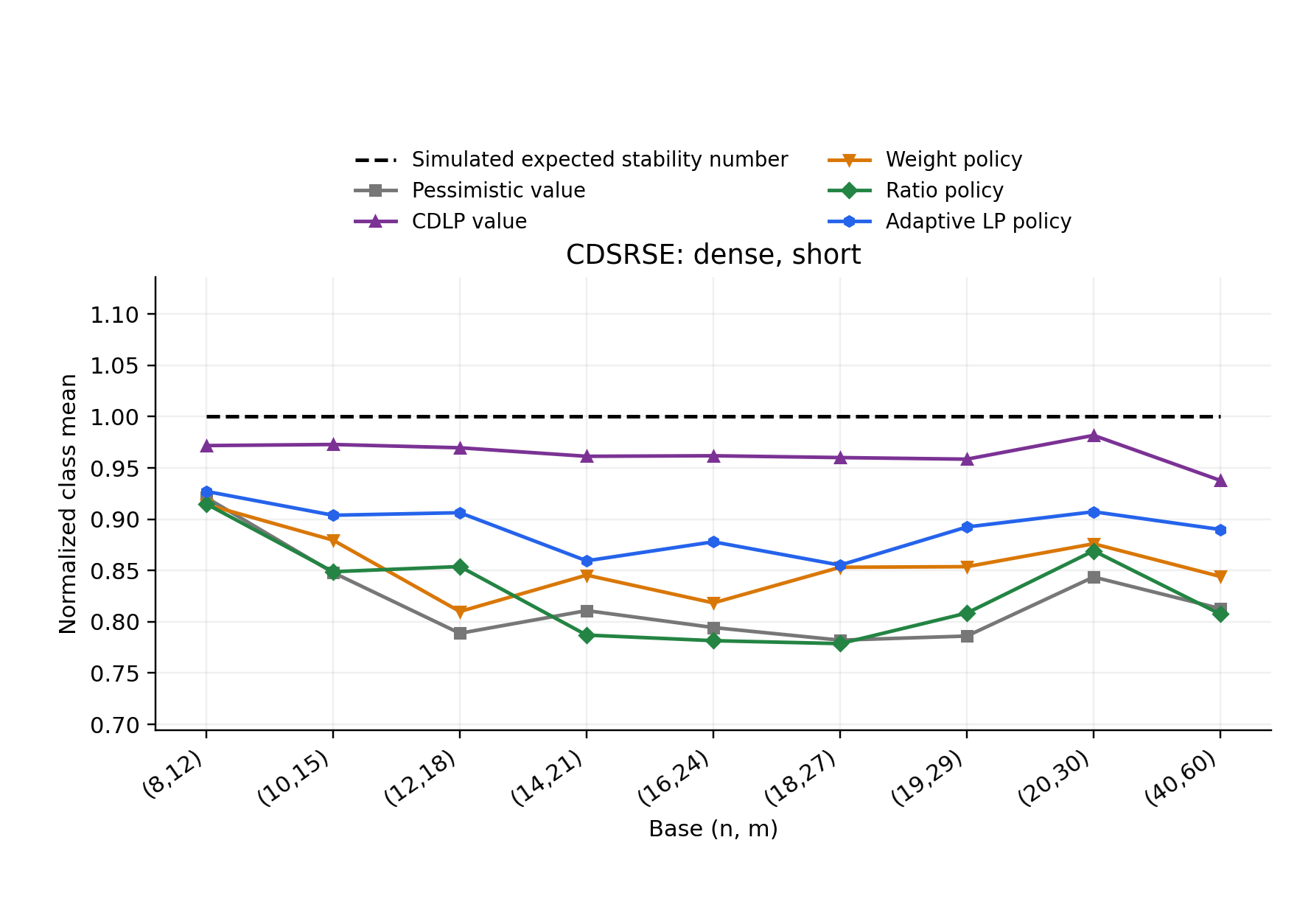}
    \caption{CDSRSE: intervals with uniform start and end times and short expected length, densely packed across the slots.}
    \label{fig:CDSRSE-Gau}
  \end{subfigure}

  \begin{subfigure}[h]{0.49\textwidth}
    \centering
    \includegraphics[scale=0.35]{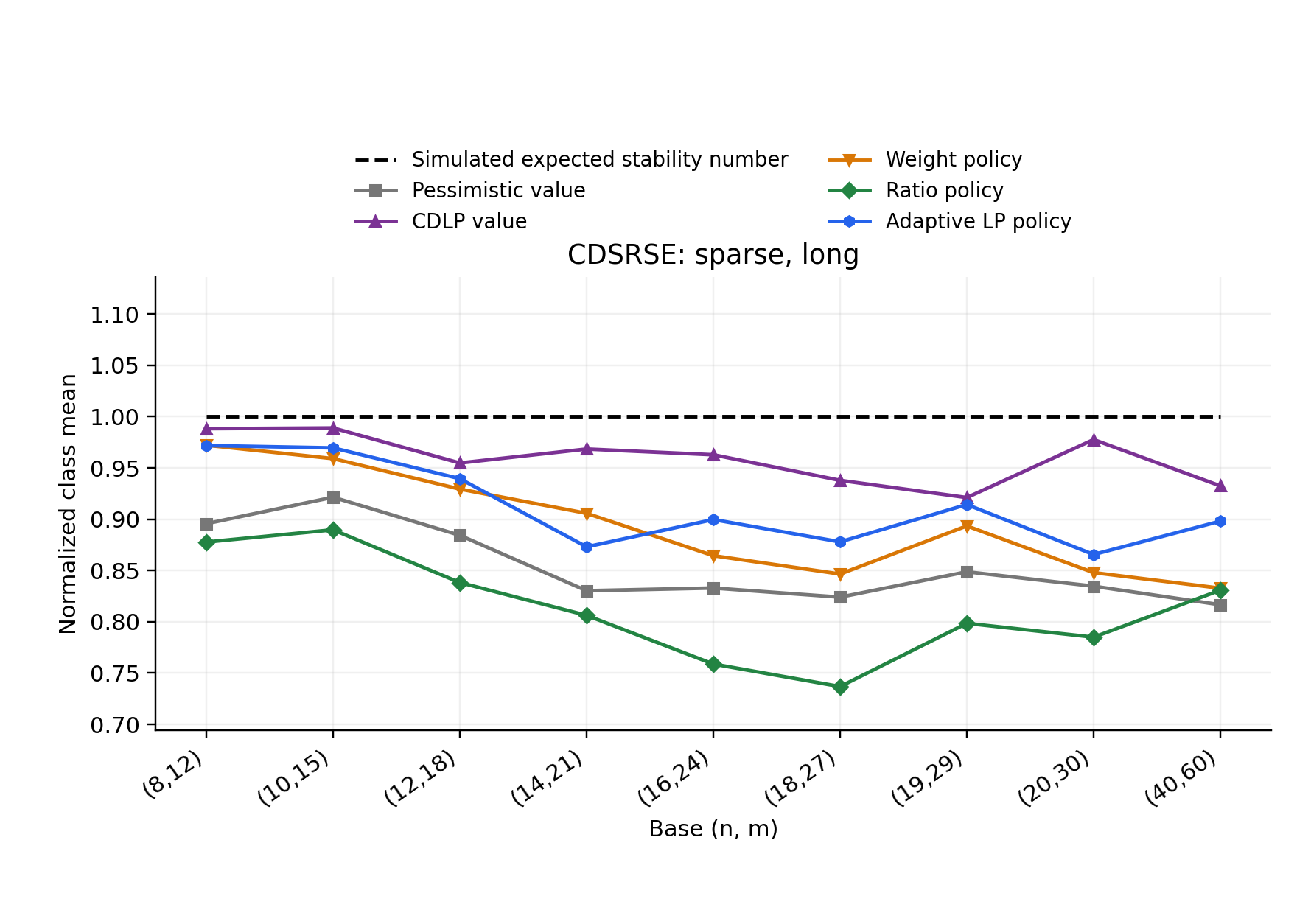}
    \caption{CDSRSE: intervals with uniform start and end times and long expected length, sparsely packed across the slots.}
    \label{fig:CDSRSE-Uni-sparse}
  \end{subfigure}
  \hfill
  \begin{subfigure}[h]{0.49\textwidth}
    \centering
    \includegraphics[scale=0.35]{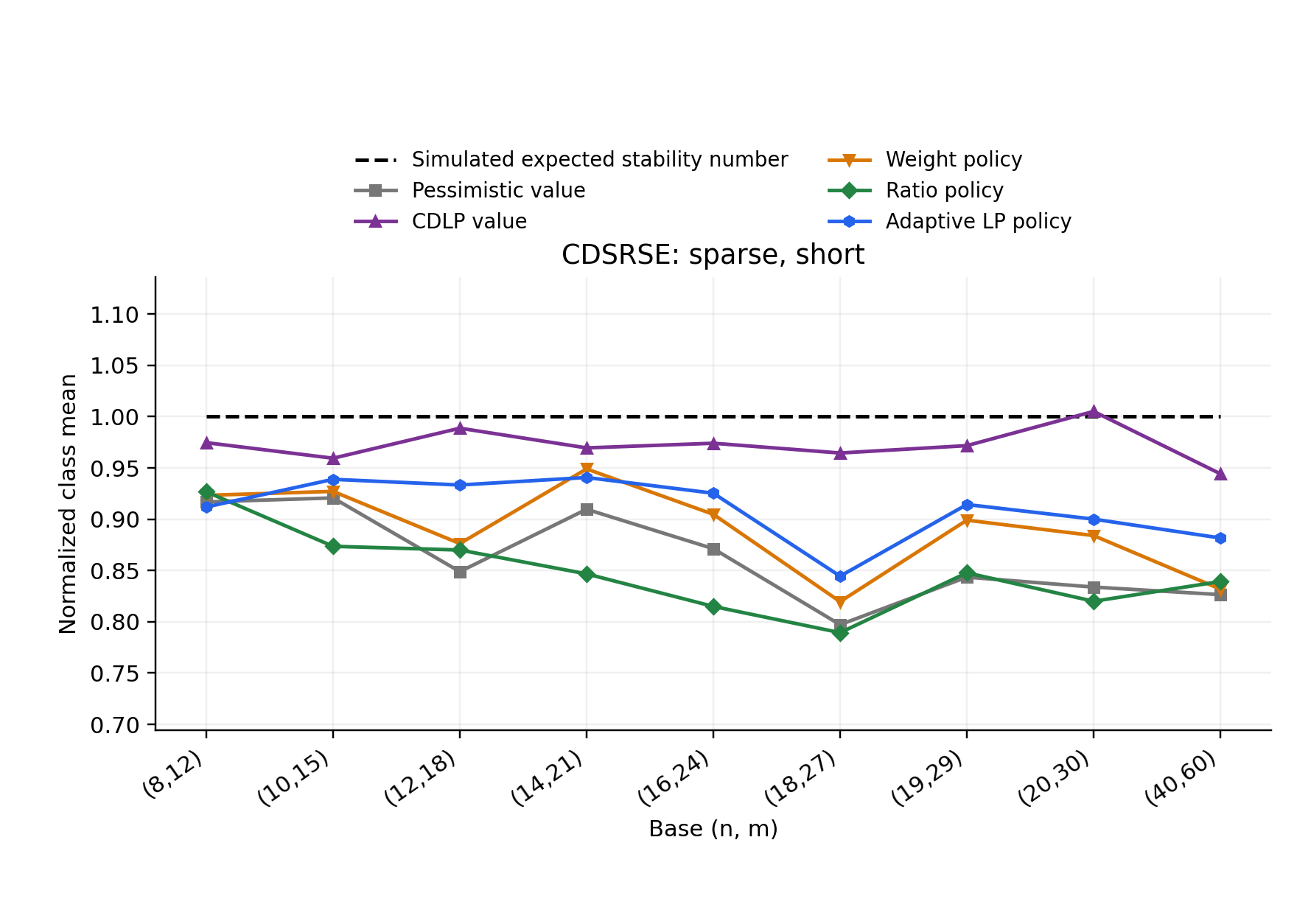}
    \caption{CDSRSE: intervals with uniform start and end times and short expected length, sparsely packed across the slots.}
    \label{fig:CDSRSE-Gau-sparse}
  \end{subfigure}
  \caption{CDSRSE}
  \label{fig:CDSRSE}
\end{figure}

\enlargethispage{2\baselineskip}
\bibliographystyle{siam}
\bibliography{SDPVFA}
\end{document}